\documentclass{amsart}

\usepackage{amsmath,amsthm,amssymb,mathrsfs,tikz,subfigure}

\numberwithin{equation}{section}

\newcommand{\ot}{\otimes}
\newcommand{\id}{\text{id}}
\newcommand{\Z}{\mathbb{Z}}

\newcommand{\N}{\mathbb{N}}
\newcommand{\C}{\mathbb{C}}
\newcommand{\G}{\mathbb{G}}

\newcommand{\aut}{\G^{aut}(B,\psi)}
\newcommand{\caut}{C(\G^{aut}(B,\psi))}
\newcommand{\um}{\frac{1}{2}}
\newcommand{\wpr}{H_{(B,\psi)}^+(\widehat{\Gamma})}
\newcommand{\cwpr}{\widehat{\Gamma}\wr_{*} \G^{aut}(B,\psi)}
\newcommand{\ccwpr}{C(\widehat{\Gamma}\wr_{*} \G^{aut}(B,\psi))}
\newcommand{\Tr}{\text{Tr}}
\newcommand{\tr}{\text{tr}}
\newcommand{\Hom}{\text{Hom}}
\newcommand{\hilb}{\text{Hilb}_f}

\theoremstyle{plain}
\newtheorem{theorem}{Theorem}[section]

\newtheorem{prop}[theorem]{Proposition}

\theoremstyle{definition}
\newtheorem{deff}[theorem]{Definition}
\newtheorem{example}[theorem]{Example}
\newtheorem{notaz}[theorem]{Notation}

\theoremstyle{remark}
\newtheorem{remark}[theorem]{Remark}

\begin{document}

\title[FREE WREATH PRODUCTS OF A DISCRETE GROUP BY $\aut$]{\textbf{The free wreath product of a discrete group by a quantum automorphism group}}
\author{Lorenzo Pittau}

\begin{abstract}
Let $\G$ be the quantum automorphism group of a finite dimensional C*-algebra $(B,\psi)$ and $\Gamma$ a discrete group. We want to compute the fusion rules of $\widehat{\Gamma}\wr_* \G$. First of all, we will revise the representation theory of $\G$ and, in particular, we will describe the spaces of intertwiners by using noncrossing partitions. It will allow us to find the fusion rules of the free wreath product in the general case of a state $\psi$. We will also prove the simplicity of the reduced C*-algebra, when $\psi$ is a trace, as well as the Haagerup property of $L^\infty(\widehat{\Gamma}\wr_* \G)$, when $\Gamma$ is moreover finite.
\end{abstract}

\maketitle

\section*{Introduction}

The core idea of the quantum automorphism group is to provide a quantum version of the standard notion of an automorphism group. As Wang proves in \cite{wan98}, given a finite dimensional C*-algebra $B$ endowed with a state $\psi$, the category of compact quantum groups acting on $B$ and $\psi$-invariant admits a universal object which is called a quantum automorphism group and is denoted $\aut$. 
In the particular case of $B=\C^n$ endowed with the uniform probability measure $\psi$, it is easy to prove that $\G^{aut}(\C^n,\psi)$ is the quantum symmetric group $S_n^+$. 
The representation theory of $\aut$ has been taken into account for the first time by Banica in \cite{ban99, ban02}: in such a paper, he showed that, if $\psi$ is a $\delta$-form and $\dim(B)\geq 4$, the spaces of intertwiners can be described by using Temperley-Lieb diagrams; moreover, the irreducible representations and the fusion rules are the same as of $SO(3)$.
Later on, in \cite{bs09} Banica and Speicher proved that, in the case of $S_n^+$ it was possible to describe the intertwiners in a simpler way by using noncrossing partitions. Bichon in \cite{bic04} introduced the definition of the free wreath product of a compact quantum group by a quantum permutation group and proved that it was possible to endow this product with a compact quantum group structure. 
Later on, in \cite{bv09} Banica and Vergnioux showed that the quantum reflection group $H^{s+}_n$ is isomorphic to $\widehat{\Z}_s\wr_* S_n^+$ and calculated its fusion rules in the case $n\geq 4$. In \cite{fr13}, Lemeux generalized this result to the free wreath product $\widehat{\Gamma}\wr_* S_n^+$, where $\Gamma$ is a discrete group and $n\geq 4$. 
In the last few years, the fusion rules of compact quantum groups have been used as a starting point for the proof of some of their properties such as simplicity, fullness, factoriality and Haagerup property.

In the present work, we aim to provide a further generalization of such results, by considering the free wreath product of the dual of a discrete group by a quantum automorphism group. The first step is revising the representation theory of $\aut$, $\dim(B)\geq 4$ by showing that, for a general $\delta$-form $\psi$, it is possible to describe the intertwiners by means of noncrossing partitions. By relying on such a graphical interpretation, we find the fusion rules and the irreducible representations of $\wpr :=\widehat{\Gamma}\wr_* \aut$, $\dim(B)\geq 4$. In this case the intertwiners are interpreted by using noncrossing partitions decorated with the elements of the discrete group $\Gamma$. In the last section we prove that, when $\psi$ is a $\delta$-trace, $C_r(H_{(B,\psi)}^+)$ is simple with unique trace if $\dim(B)\geq 8$, while $L^\infty(\widehat{\Gamma}\wr_* \G)$ has the Haagerup property if $\Gamma$ is finite.
All these results are also presented in a more general form, by dropping the $\delta$ condition for $\psi$ and by proving that in this case the free wreath product can be decomposed as a free product of quantum groups in which the condition still holds.

\section*{Acknowledgements}
The author is very grateful to Pierre Fima for the numerous discussions they had about the subjects of this paper and for all his suggestions. The author would like to thank Teodor Banica and Fran{\c c}ois Lemeux for useful discussions. The author also thanks Roland Vergnioux for all his commentaries.

\section{The quantum automorphism group $\aut$}
In this section we start by recalling some known results concerning the quantum automorphism group and the category of noncrossing partitions. Then, we will give a new description of the intertwining spaces of a quantum automorphism group, by using noncrossing partitions instead of Temperley-Lieb diagrams.
\begin{notaz}
The symbol $\ot$ will be used to denote the tensor product of Hilbert spaces, the minimal tensor product of C*-algebras or the tensor product of von Neumann algebras, depending on the context. We denote by $\mathcal{L}(H)$ the C*-algebra of bounded linear operators on an Hilbert space $H$.
\end{notaz}
Let $(e_i)_i$ be an orthonormal basis of the Hilbert space $H$ and let $(e_{ij})_{ij}$ be the associated matrix units. An element $u\in\mathcal{L}(H)\ot A$, where $A$ is a unital C*-algebra, can be written as $u=\sum e_{ij}\ot u_{ij}$. The $u_{ij}\in A$ are called the coefficients of $u$.

A Woronowicz compact quantum group (see \cite{wor87,wor98}) is a pair $(C(\G),\Delta)$ where $C(\G)$ is a unital C*-algebra and $\Delta:C(\G)\longrightarrow C(\G)\otimes C(\G)$ a $\ast$-homomorphism together with a family of unitaries $(u^{\alpha})_{\alpha\in I}$, $u^{\alpha}\in \mathcal{L}(H_\alpha)\ot C(\G)$, where $H_\alpha$ is a finite dimensional Hilbert space, such that:
\begin{itemize}
\item the $*$-subalgebra generated by the coefficients of all the $u^\alpha$, $\alpha\in I$, is dense in $C(\G)$,
\item for all $\alpha\in I$ we have $(\id\ot\Delta)(u^\alpha)=u_{12}^\alpha u_{13}^\alpha$,
\item for all $\alpha\in I$ the transposed matrix $(u^\alpha)^t$ is invertible.
\end{itemize}

A finite dimensional representation of the compact quantum group $(C(\G),\Delta)$ is an element $v\in \mathcal{L}(H)\ot C(\G)$, where $H$ is a finite dimensional Hilbert space, such that $(\id\ot\Delta)(v)=v_{12} v_{13}$.
Given two representations $v\in \mathcal{L}(H)\ot C(\G)$ and $w\in \mathcal{L}(K)\ot C(\G)$, the space of the intertwiners is $\Hom(v,w)=\{T\in \mathcal{L}(H,K)|$ $(T\otimes 1)v=w(T\otimes 1)\}$.

We recall that every finite dimensional C*-algebra $B$ is isomorphic to a multimatrix C*-algebra. In what follows, we will consider the decomposition $B=\bigoplus_{\alpha=1}^c M_{n_\alpha}(\C)$. Let $\mathscr{B}=\{(e_{ij}^\alpha)_{i,j=1,\dots,n_\alpha},\alpha=1,\dots,c\}$ be the canonical basis of matrix units. Suppose $B$ endowed with the usual multiplication $m:B\otimes B\longrightarrow B$, $m(e_{ij}^\alpha \otimes e_{kl}^\beta)=\delta_{jk}\delta_{\alpha\beta} e_{il}^\alpha$ and unity $\eta:\C\longrightarrow B$, $\eta(1)=\sum_{\alpha=1}^c\sum_{i=1}^{n_\alpha}e_{ii}^\alpha$.
Every finite dimensional C*-algebra $B$ can be endowed with a Hilbert space structure by considering the scalar product $\langle x,y\rangle =\psi(y^*x)$ which is induced by a faithful state $\psi$.
\begin{deff}
Let $B$ be a finite dimensional C*-algebra and $\delta>0$. A faithful state $\psi:B\longrightarrow \C$ is a $\delta$-form, if the multiplication map of $B$ and its adjoint with respect to the inner product induced by $\psi$ satisfy $mm^*=\delta\cdot \id_B$. If such a $\psi$ is also a trace, then it is called a tracial $\delta$-form or a $\delta$-trace.
\end{deff}

\begin{remark}\label{df}
For every faithful state $\psi:M_n(\C)\longrightarrow\C$, there exists $Q\in M_n(\C)$, $Q>0$, $\Tr(Q)=1$ such that $\psi=\Tr(Q\cdot)$. In particular, we notice that every such $\psi$ is a $\delta$-form, with $\delta=\Tr(Q^{-1})$.
Broadly speaking, if $B=\bigoplus_{\alpha=1}^c M_{n_\alpha}(\C)$, then every faithful state $\psi$ over $B$ is of the form $\psi=\bigoplus_{\alpha=1}^c \Tr(Q_\alpha\cdot)$ for a suitable family $Q_\alpha\in M_{n_\alpha}(\C)$, $Q_\alpha>0$, $\sum_\alpha \Tr(Q_\alpha)=1$. In particular, if $\Tr(Q_\alpha^{-1})=\delta$ for all $\alpha$, then $\psi$ is a $\delta$-form.
\end{remark}

It is well known that every positive complex matrix is diagonalizable. In particular, the matrices $Q_\alpha$ are always similar to diagonal matrices with positive real eigenvalues. In what follows, we will always assume that the matrix units $e_{ij}^\alpha$ are associated with a basis of $\C^{n_\alpha}$ which diagonalizes $Q_\alpha$. 
We will denote $Q_{i,\alpha}$ the eigenvalue in position $(i,i)$ of $Q_\alpha$ written with respect to this fixed diagonalizing basis. We observe that $\psi(e_{ij}^\alpha)=\Tr(Q_\alpha e_{ij}^\alpha)=\delta_{ij}Q_{i,\alpha}$. Then, the basis $\mathscr{B}$ is always orthogonal with respect to the scalar product induced by $\psi$. By normalizing $\mathscr{B}$ we obtain the orthonormal basis $$\mathscr{B}'=\{b_{ij}^\alpha | b_{ij}^\alpha=\psi(e_{jj}^\alpha)^{-\um} e_{ij}^\alpha,\; i,j=1,\dots,n_\alpha,\;\alpha=1,\dots,c\}.$$

We are now ready to give one of the main definitions of the paper. The quantum automorphism group of a finite dimensional C*-algebra $(B,\psi)$ is the universal object in the category of the compact quantum groups acting on $B$ and leaving the state $\psi$ invariant (see \cite{ban99, wan98}). More precisely it can be defined as follows.

\begin{deff}\label{autb}
Let $B$ be a finite dimensional C*-algebra with multiplication $m:B\otimes B\longrightarrow B$ and unity $\eta:\C\longrightarrow B$. Let $\psi$ be a state on $B$. Let $\caut$ be the universal unital C*-algebra generated by the coefficients of $u\in\mathcal{L}(B)\ot \caut$ with the following relations
\begin{itemize}
\item $u$ is unitary,
\item $m\in \Hom(u^{\otimes 2},u)$,
\item $\eta\in \Hom(1,u)$.
\end{itemize}
Then $\caut$ endowed with the unique comultiplication $\Delta$ such that $(\id\ot\Delta)(u)=u_{12}u_{13}$ is the quantum automorphism group of the C*-algebra $(B,\psi)$ and will be denoted $\aut$.
\end{deff}

The representation theory of $\aut$ is well known from \cite{ban02}, but, in order to generalize it to the free wreath product, we need a description of the intertwining spaces in terms of noncrossing partitions, as it was done for $\G^{aut}(\C^n,\tr)$ (see \cite{bs09}). For this aim, we recall some standard definitions about noncrossing partitions.

\begin{deff}
Let $k,l\in\N$. We use the notation $NC(k,l)$ to denote the set of noncrossing partitions between $k$ upper points and $l$ lower points. These noncrossing partitions can be represented by diagrams obtained by connecting $k$ points lying on an upper imaginary line and $l$ points lying on a lower line in all the possible ways by using strings that pass only in the part of the plane which lies between the two rows of points and that do not cross each other. A block in a noncrossing partition is a set of points connected to each other. A point without connections can be considered as a trivial block. The total number of blocks of $p\in NC(k,l)$ is denoted $b(p)$.
\end{deff}

\begin{deff}\label{opdiag}
Let $p\in NC(k,l), q\in NC(v,w)$. We define the follow diagram operations:
\begin{enumerate}
\item the tensor product $p\otimes q$ is the diagram in $NC(k+v,l+w)$ obtained by horizontal concatenation of the diagrams $p$ and $q$,
\item if $l=v$ it is possible to define the composition $qp$ as the diagram in $NC(k,w)$ obtained by identifying the lower points of $p$ with the upper points of $q$ and by removing all the blocks which have possibly appeared and which contain neither one of the upper points of $p$ nor one of the lower points of $q$; such operation, when it is defined, is associative,
\item the adjoint $p^*$ is the diagram in $NC(l,k)$ obtained by reflecting the diagram $p$ with respect to an horizontal line which lies between the rows of the upper and of the lower points.
\end{enumerate}
\end{deff}

\begin{notaz}
The blocks possibly removed when multiplying two noncrossing partitions $p\in NC(k,l)$, $q\in NC(l,w)$ will be called central blocks and their number denoted by $cb(p,q)$. 
Furthermore, the vertical concatenation can produce some (closed) cycles which will not appear in the final noncrossing partition either. Intuitively, they are the rectangles which are obtained when two or more central points are connected both in the upper and in the lower noncrossing partition (see the example below). 
In a more formal way, the number of cycles, denoted $cy(p,q)$, is defined as 
\begin{equation*}
cy(p,q)=l+b(qp)+cb(p,q)-b(p)-b(q)
\end{equation*}
\end{notaz}

\begin{example}
In order to clarify the multiplication operation and the concepts of block, central block as well as cycle, we can think of $p\in NC(4,17)$ and $q\in NC(17,5)$ in the following example:

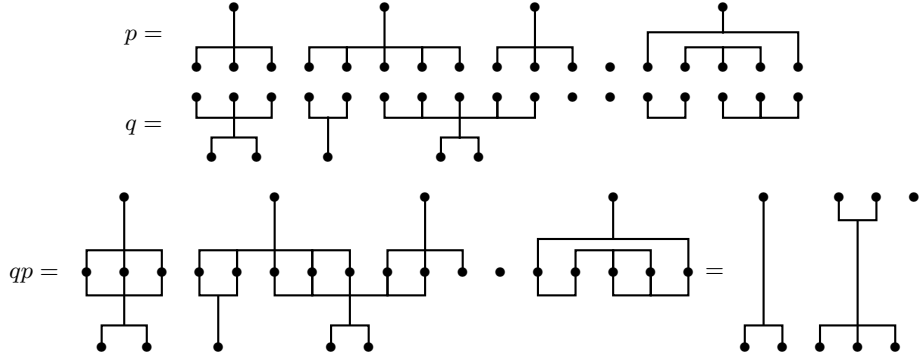
{\centering
 \begin{tikzpicture}[thick,font=\small]
    \path (.3,.6) node{$p=$}
    		  (.3,-.6) node{$q=$}
    		  (1.5,1) node{$\bullet$} node[above](a){}
    		  (3.5,1) node{$\bullet$} node[above](b){}
    		  (5.5,1) node{$\bullet$} node[above](c){}
          (1,.2) node{$\bullet$} node[below](d){}
          (1.5,.2) node{$\bullet$} node[below](e){}
          (2,.2) node{$\bullet$} node[below](f){}
          (2.5,.2) node{$\bullet$} node[below](g){}
          (3,.2) node{$\bullet$} node[below](h){}
          (3.5,.2) node{$\bullet$} node[below](i){}
          (4,.2) node{$\bullet$} node[below](l1){}
          (4.5,.2) node{$\bullet$} node[below](m1){}
          (5,.2) node{$\bullet$} node[below](n1){}
          (5.5,.2) node{$\bullet$} node[below](o){}
          (6,.2) node{$\bullet$} node[below](p){}
          (6.5,.2) node{$\bullet$} node[below](q){}
          (1,-.2) node{$\bullet$} node[above](dd){}
          (1.5,-.2) node{$\bullet$} node[above](ee){}		  
          (2,-.2) node{$\bullet$} node[above](ff){}
          (2.5,-.2) node{$\bullet$} node[above](gg){}
          (3,-.2) node{$\bullet$} node[above](hh){}
          (3.5,-.2) node{$\bullet$} node[above](ii){}
          
          (1.2,-1) node{$\bullet$} node[below](l){}
          (1.8,-1) node{$\bullet$} node[below](m){}
          (2.75,-1) node{$\bullet$} node[below](n){}
          (4.25,-1) node{$\bullet$} node[below](r){}
          (4.75,-1) node{$\bullet$} node[below](s){}
          
          (4,-.2) node{$\bullet$} node[above](ll1){}
          (4.5,-.2) node{$\bullet$} node[above](mm1){}
          (5,-.2) node{$\bullet$} node[above](nn1){}
          (5.5,-.2) node{$\bullet$} node[above](oo){}
          (6,-.2) node{$\bullet$} node[above](pp){}
          (6.5,-.2) node{$\bullet$} node[above](qq){}
          
          (7,.2) node{$\bullet$} node[below](t){}
          (7.5,.2) node{$\bullet$} node[below](u){}
          (8,.2) node{$\bullet$} node[below](v){}
          (8.5,.2) node{$\bullet$} node[below](w){}
          (9,.2) node{$\bullet$} node[below](x){}
          
          (7,-.2) node{$\bullet$} node[above](tt){}
          (7.5,-.2) node{$\bullet$} node[above](uu){}
          (8,-.2) node{$\bullet$} node[above](vv){}
          (8.5,-.2) node{$\bullet$} node[above](ww){}
          (9,-.2) node{$\bullet$} node[above](xx){}
          (8,1) node{$\bullet$} node[above](z){};

    \draw (d) -- +(0,+0.4) -| (f);
    \draw (h) -- +(0,+0.4) -| (l1);
    \draw (l1) -- +(0,+0.4) -| (m1);
    \draw (n1) -- +(0,+0.4) -| (o);
    \draw (o) -- +(0,+0.4) -| (p);
    \draw (ii) -- +(0,-0.4) -| (ll1);
    \draw (ll1) -- +(0,-0.4) -| (nn1);
    \draw (r) -- +(0,+0.4) -| (s);
    \draw (nn1) -- +(0,-0.4) -| (oo);
    \draw (e) -- (a);
    \draw (o) -- (c);
    \draw (b) -- (i);
    \draw (mm1) -- +(0,-0.68);
    \draw (g) -- +(0,+0.4) -| (h);
    \draw (dd) -- +(0,-0.4) -| (ff);
    \draw (l) -- +(0,+0.4) -| (m);
    \draw (1.5,-.2) -- (1.5,-0.73);
    \draw (gg) -- +(0,-0.4) -| (hh);
    \draw (n) -- +(0,0.68);
    \draw (t) -- +(0,+0.6) -| (x);
    \draw (u) -- +(0,+0.4) -| (v);
    \draw (v) -- +(0,+0.4) -| (w);
    \draw (tt) -- +(0,-0.4) -| (uu);
    \draw (vv) -- +(0,-0.4) -| (ww);
    \draw (ww) -- +(0,-0.4) -| (xx);
    \draw (z) -- +(0,-0.48);

 \end{tikzpicture}
 
  \begin{tikzpicture}[thick,font=\small]
    \path (1.5,1) node{$\bullet$}
          (1,0) node{$\bullet$}
          (1.5,0) node{$\bullet$}
          (2,0) node{$\bullet$}
          (2.5,0) node{$\bullet$}
          (3,0) node{$\bullet$}
          (3.5,0) node{$\bullet$}
          (1.2,-1) node{$\bullet$}
          (1.8,-1) node{$\bullet$}
          (2.75,-1) node{$\bullet$}
          (.3,-0.03) node{$qp=$}
          (9.35,0) node{$=$}
          (10,1) node{$\bullet$}node[above](a){}
          (11,1) node{$\bullet$}node[above](b){}
          (11.5,1) node{$\bullet$}node[above](c){}
          (9.75,-1) node{$\bullet$}node[below](e){}
          (10.25,-1) node{$\bullet$}node[below](f){}
          (10.75,-1) node{$\bullet$}node[below](g){}
          (11.25,-1) node{$\bullet$}node[below](h){}
          (11.75,-1) node{$\bullet$}node[below](i){}
          (12,1) node{$\bullet$}
          
          (4,0) node{$\bullet$}
          (4.5,0) node{$\bullet$}
          (5,0) node{$\bullet$}
          (5.5,0) node{$\bullet$}
          (6,0) node{$\bullet$}
          (6.5,0) node{$\bullet$}
          (3.5,1) node{$\bullet$}
          (5.5,1) node{$\bullet$}
          (4.25,-1) node{$\bullet$}
          (4.75,-1) node{$\bullet$}
          
          (8,1) node{$\bullet$}
          (7,0) node{$\bullet$}
          (7.5,0) node{$\bullet$}
          (8,0) node{$\bullet$}
          (8.5,0) node{$\bullet$}
          (9,0) node{$\bullet$};
          
    \draw (9.75,-1) -- +(0,+0.3) -| (10.25,-1);
    \draw (10.75,-1) -- +(0,+0.3) -| (11.75,-1);
    \draw (11,1) -- +(0,-0.3) -| (11.5,1);
    \draw (10,1) -- +(0,-1.7);
    \draw (11.25,-1) -- +(0,1.7);
    \draw (8,1) -- +(0,-0.55); 
    \draw (7,0) -- +(0,+0.45) -| (9,0);
    \draw (8,0) -- +(0,+0.3) -| (8.5,0);
    \draw (7.5,0) -- +(0,+0.3) -| (8,0);
    \draw (7,0) -- +(0,-0.3) -| (7.5,0);
    \draw (8,0) -- +(0,-0.3) -| (8.5,0);
    \draw (8.5,0) -- +(0,-0.3) -| (9,0);
    \draw (1,0) -- +(0,+0.3) -| (2,0);
    \draw (1.5,0) -- (1.5,1);
    \draw (2.5,0) -- +(0,+0.3) -| (3,0);
    \draw (1,0) -- +(0,-0.3) -| (2,0);
    \draw (1.2,-1) -- +(0,+0.3) -| (1.8,-1);
    \draw (1.5,0) -- (1.5,-0.7);
    \draw (2.5,0) -- +(0,-0.3) -| (3,0);
    \draw (2.75,-1) -- +(0,0.7);   
    \draw (3,0) -- +(0,+0.3) -| (4,0);
    \draw (4,0) -- +(0,+0.3) -| (4.5,0);
    \draw (5,0) -- +(0,+0.3) -| (6,0);
    \draw (3.5,0) -- +(0,-0.3) -| (4,0);
    \draw (4,0) -- +(0,-0.3) -| (5,0);
    \draw (5,0) -- +(0,-0.3) -| (5.5,0);
    \draw (4.25,-1) -- +(0,+0.3) -| (4.75,-1);
    \draw (3.5,1) -- (3.5,0);
    \draw (5.5,1) -- (5.5,0);
    \draw (4.5,0) -- +(0,-0.7);

 \end{tikzpicture}
\par}
\noindent
We suddenly have $b(p)=6$, $b(q)=7$, $b(qp)=3$ and $cb(p,q)=1$. Then, the number of cycles is $cy(p,q)=17+3+1-6-7=8$.
\end{example}

In the following proposition, we introduce a simple but useful relation concerning the number of cycles obtained by multiplying three noncrossing partitions.

\begin{prop}\label{propcy}
Let $p\in NC(k,l)$, $r\in NC(l,m)$ and $s\in NC(m,v)$. Then the following relation holds:
\begin{equation}\label{relcy}
cy(p,sr)=cy(p,r)+cy(rp,s)-cy(r,s)
\end{equation}
\end{prop}

\begin{proof}
By making use of the definition of a cycle and of the associativity of the composition, the relation (\ref{relcy}) reduces to $cb(p,sr)=cb(p,r)+cb(rp,s)-cb(r,s)$. In order to complete the proof, it is then enough to observe that $cb(p,sr)=cb(p,r)$, because the number of central blocks obtained by concatenating $p$ and $sr$ does not depend on the noncrossing partition $s$; similarly $cb(rp,s)=cb(r,s)$.
\end{proof}

The next goal is to assign a linear map to every noncrossing partition. In order to define such a map, we will make use of the following notation which generalizes the classical one.

\begin{notaz}\label{coeff}
Consider a diagram $p\in NC(k,l)$ and associate to every point a matrix unit of the canonical basis of $B$. Let $(b_{i_1j_1}^{\alpha_1},\dots,b_{i_kj_k}^{\alpha_k})$ be the ordered set of elements associated to the upper points and $(b_{r_1s_1}^{\beta_1},\dots,b_{r_ls_l}^{\beta_l})$ the elements associated to the lower points. Let $ij$ and $\alpha$ be the multi-index notation for the indices of these matrices, in particular $ij=((i_1,j_1),\dots,(i_k,j_k))$ and $\alpha=(\alpha_1,\dots,\alpha_k)$; in a similar way, we define $rs$ and $\beta$.\\
Denote by $b_v,v=1,\dots,m$ the different blocks of $p$, and let $b_v^\uparrow$ ($b_v^\downarrow$) be the ordered product of the matrix units associated to the upper (lower) points of the block $b_v$. Such a product is conventionally the identity matrix, if there are no upper (lower) points in the block. Define
\begin{equation}\label{coeffd}
\delta_p^{\alpha,\beta}(ij,rs):=\prod_{v=1}^m \psi((b_v^\downarrow)^*b_v^\uparrow)
\end{equation}
\end{notaz}

\begin{example}
Consider the following noncrossing partition $p$ in which we associated an element of the basis to every point.

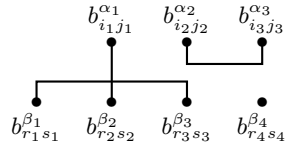
{\centering
 \begin{tikzpicture}[thick,font=\small]
    \path (2,1) node{$\bullet$} node[above](a){$b_{i_1j_1}^{\alpha_1}$}
          (1,.2) node{$\bullet$} node[below](d){$b_{r_1s_1}^{\beta_1}$}
          (2,.2) node{$\bullet$} node[below](e){$b_{r_2s_2}^{\beta_2}$}
          (3,.2) node{$\bullet$} node[below](f){$b_{r_3s_3}^{\beta_3}$}
          (3,1) node{$\bullet$} node[above](g){$b_{i_2j_2}^{\alpha_2}$}
          (4,1) node{$\bullet$} node[above](h){$b_{i_3j_3}^{\alpha_3}$}
          (4,.2) node{$\bullet$} node[below](i){$b_{r_4s_4}^{\beta_4}$};
          
    \draw (d) -- +(0,+0.6) -| (f);
    \draw (e) -- (a);
    \draw (g) -- +(0,-0.6) -|(h);
 \end{tikzpicture}
\par}
\noindent
In this case the coefficient just introduced is $$\delta_p^{\alpha,\beta}(ij,rs)=\psi((b_{r_1s_1}^{\beta_1}b_{r_2s_2}^{\beta_2}b_{r_3s_3}^{\beta_3})^*b_{i_1j_1}^{\alpha_1})\psi(b_{i_2j_2}^{\alpha_2}b_{i_3j_3}^{\alpha_3})\psi((b_{r_4s_4}^{\beta_4})^*)$$
\end{example}

\begin{deff}\label{map2}
We associate to every element $p\in NC(k,l)$ the linear map $T_p:B^{\otimes k}\longrightarrow B^{\otimes l}$ which is defined by:
$$T_p(b_{i_1j_1}^{\alpha_1} \otimes\cdots \otimes b_{i_kj_k}^{\alpha_k})=\sum_{r,s,\beta}\delta_p^{\alpha,\beta}(ij,rs)b_{r_1s_1}^{\beta_1} \otimes\cdots\otimes b_{r_ls_l}^{\beta_l}$$
\end{deff}

\begin{example}\label{exmap}
The diagram $p$ which is associated to the multiplication map $m$ (writing explicitly  $b_{i_1,j_1}^{\alpha_1}, b_{i_2,j_2}^{\alpha_2}$ on the upper points and $b_{r_1,s_1}^{\beta_1}$ on the lower point) is:

{\centering
 \begin{tikzpicture}[thick,font=\small]
    \path (0,.8) node{$\bullet$} node[above](a) {$b_{i_1 j_1}^{\alpha_1}$}
          (1,.8) node{$\bullet$} node[above](b) {$b_{i_2 j_2}^{\alpha_2}$}
          (.5,0) node{$\bullet$} node[below](c) {$b_{r_1 s_1}^{\beta_1}$};

    \draw (c) -- +(0,0.78);
    \draw (a) -- +(0,-0.65) -| (b);
 \end{tikzpicture}
\par}
\noindent
Here, by applying the definition
$$\left.\begin{array}{lll}
\delta_p^{\alpha,\beta}{((i_1,j_1,i_2,j_2),(r_1,s_1))}&=&
\psi((\psi(e_{s_1 s_1}^{\beta_1})^{-\um}e_{r_1 s_1}^{\beta_1})^*\psi(e_{j_1 j_1}^{\alpha_1})^{-\um}e_{i_1 j_1}^{\alpha_1}\psi(e_{j_2 j_2}^{\alpha_2})^{-\um}e_{i_2 j_2}^{\alpha_2})\\ &=& 

\psi(e_{s_1 s_1}^{\beta_1})^{-\um}\psi(e_{j_1 j_1}^{\alpha_1})^{-\um}\psi(e_{j_2 j_2}^{\alpha_2})^{-\um}\psi(e_{s_1 r_1}^{\beta_1}e_{i_1 j_1}^{\alpha_1}e_{i_2 j_2}^{\alpha_2})\\ &=& 
\delta_{\beta_1,\alpha_1}\delta_{\alpha_1,\alpha_2}\delta_{r_1,i_1}\delta_{j_1,i_2}\delta_{s_1,j_2}\psi(e_{j_1 j_1}^{\alpha_1})^{-\um}
\end{array}\right.$$
so the associated map $T_p:B^{\otimes 2}\longrightarrow B$ is given by $$T_p(b_{i_1j_1}^{\alpha_1}\otimes b_{i_2j_2}^{\alpha_2})=\delta_{\alpha_1,\alpha_2}\delta_{j_1,i_2}\psi(e_{j_1 j_1}^{\alpha_1})^{-\um}b_{i_1j_2}^{\alpha_1}$$ which is the multiplication $m$.
\end{example}

Then, we have the following compatibility result (see \cite[Proposition 1.9]{bs09} for the case of $\C^n$ with the canonical trace).

\begin{prop}\label{propmap}
Let $p\in NC(l,k), q\in NC(v,w)$. We have:
\begin{enumerate}
\item $T_{p\otimes q}=T_p\otimes T_q$,
\item $T_p^*=T_{p^*}$,
\item if $k=v$ then $T_{qp}=\delta^{-cy(p,q)}T_q T_p$.
\end{enumerate}
\end{prop}

\begin{proof}
Even if the core of the proof is essentially the same as \cite{bs09}, it is necessary to pay much more attention to the computations.\\
The relation 1 is clear because $\delta_p^{\alpha,\beta}(ij,rs)\delta_q^{\alpha',\beta'}(IJ,RS)=\delta_{p\otimes q}^{\alpha\alpha',\beta\beta'}(ijIJ,rsRS)$.\\
The relation 2 follows from $\delta_p^{\alpha,\beta}(ij,rs)=\delta_{p^*}^{\beta,\alpha}(rs,ij)$, which is true because $\psi((b_{ij}^{\alpha})^*)$ $=\psi(b_{ij}^{\alpha})$ (we point out that the matrix associated to $\psi$ has real eigenvalues).\\
The relation 3 is more difficult to prove true. In particular, we need to show that
\begin{equation}\label{coeffcomp}
\sum_\beta\sum_{r,s=1}^{n_\beta}\delta_p^{\alpha,\beta}(ij,rs)\delta_q^{\beta,\gamma}(rs,RS)=\delta^{cy(p,q)}\delta_{qp}^{\alpha,\gamma}(ij,RS)
\end{equation}
\noindent
We remind that every noncrossing partition is obtained by using compositions, tensor products and adjoints of the basic morphisms $m,\eta$ and $\id_B$.
In order to prove the composition formula between the maps associated to two noncrossing partitions $p$ and $q$, we can think of decomposing $q$ in the composition of a sequence of noncrossing partitions corresponding to elementary maps of type $\id_B^{\otimes u}\otimes f \otimes \id_B^{\otimes v}$ where $u,v\in\N$ and $f=\id_B,m,m^*,\eta,\eta^*$.

If we suppose that relation 3 holds for the composition of a general noncrossing partition $p$ with this kind of maps and let $q=q_s\cdots q_2q_1$ be such a decomposition of $q$, then we have

$$\left.\begin{array}{lll}
T_{qp}&=&T_{q_s \cdots q_2q_1p}\\ &=&\delta^{-cy(q_{s-1}\cdots q_1p,q_s)}T_{q_s}T_{q_{s-1}\cdots q_1p}\\ &=&
\delta^{-cy(q_{s-1}\cdots q_1p,q_s)-cy(q_{s-2}\cdots q_1p,q_{s-1})-\cdots -cy(p,q_1)}T_{q_s}T_{q_{s-1}}\cdots T_{q_1}T_p\\ &=&

\delta^{-cy(q_{s-1}\cdots q_1p,q_s)-cy(q_{s-2}\cdots q_1p,q_{s-1})-\cdots -cy(p,q_1)+cy(q_1,q_2)}   T_{q_s}\cdots T_{q_2q_1}T_p\\ &=&

\delta^{-cy(q_{s-1}\cdots q_1p,q_s)-\cdots -cy(p,q_1)+cy(q_1,q_2)+\cdots +cy(q_{s-1}\cdots q_1,q_s)}  T_{q_s\cdots q_2q_1}T_p\\ &=&

\delta^{-cy(p,q_s\cdots q_1)}T_{q_s\cdots q_1}T_p\\ &=&
\delta^{-cy(p,q)}T_{q}T_p
\end{array}\right.$$
where the second to last equality follows by applying $s$ times Proposition \ref{propcy}.

Now we only need to show the relation (\ref{coeffcomp}) when composing the map associated to $p$ with an elementary map. Furthermore, we can consider that the noncrossing partition $p$ has only one block (or possibly two when $f=m$); indeed it is possible to reconstruct the multi-block case by using a tensor product argument or by generalizing the following proof in an obvious way.
Let's start now the computations in the different cases. Let $p\in NC(l,k)$ be a one block noncrossing partition. The first case we take into account is the composition of $T_p$ with $T_q=\id_B^{\otimes k}$. The corresponding diagram is the following:

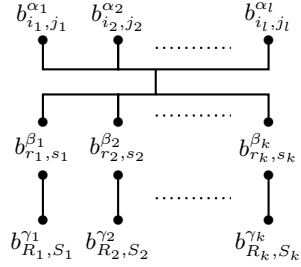
{\centering
 \begin{tikzpicture}[thick,font=\small]
    \path 
    		  (1,1.5) node{$\bullet$} node[above](a){$b_{i_1,j_1}^{\alpha_1}$}
		  (2,1.5) node{$\bullet$} node[above](b){$b_{i_2,j_2}^{\alpha_2}$}
		  (4,1.5) node{$\bullet$} node[above](c){$b_{i_l,j_l}^{\alpha_l}$}          
          
          (1,.4) node{$\bullet$} node[below](d){$b_{r_1,s_1}^{\beta_1}$}
          (2,.4) node{$\bullet$} node[below](e){$b_{r_2,s_2}^{\beta_2}$}
          (4,.4) node{$\bullet$} node[below](f){$b_{r_k,s_k}^{\beta_k}$}

          (1,-.3) node{$\bullet$} node[above](dd){}
          (2,-.3) node{$\bullet$} node[above](ee){}
          (4,-.3) node{$\bullet$} node[above](ff){}
          
          (1,-.9) node{$\bullet$} node[below](g){$b_{R_1,S_1}^{\gamma_1}$}
          (2,-.9) node{$\bullet$} node[below](h){$b_{R_2,S_2}^{\gamma_2}$}
          (4,-.9) node{$\bullet$} node[below](i){$b_{R_k,S_k}^{\gamma_k}$};

    \draw (a) -- +(0,-0.7) -| (b);
    \draw (b) -- +(0,-0.7) -| (c);
    \draw (d) -- +(0,+0.7) -| (e);
    \draw (e) -- +(0,+0.7) -| (f);
    \draw (dd) -- (g);
    \draw (ee) -- (h);
    \draw (ff) -- (i);
    \draw (2.5,1.12) -- (2.5,0.77);
     \draw [dotted] (2.5,1.4) -- (3.5,1.4);
     \draw [dotted] (2.5,.5) -- (3.5,.5);
     \draw [dotted] (2.5,-.6) -- (3.5,-.6);
     
\end{tikzpicture}
\par}
\noindent
In this case relation (\ref{coeffcomp}) is satisfied, indeed
\begin{eqnarray*}
\sum_\beta\sum_{r,s=1}^{n_\beta}\delta_p^{\alpha,\beta}(ij,rs)\delta_q^{\beta,\gamma}(rs,RS)
 & =&\sum_\beta\sum_{r,s=1}^{n_\beta}\psi((b_{r_1,s_1}^{\beta_1}\cdots b_{r_k,s_k}^{\beta_k})^*(b_{i_1,j_1}^{\alpha_1}\cdots b_{i_l,j_l}^{\alpha_l}))
\prod_{t=1}^k\psi((b_{R_t,S_t}^{\gamma_t})^*b_{r_t,s_t}^{\beta_t})\\
 &=&\sum_\beta\sum_{r,s=1}^{n_\beta}\psi((b_{r_1,s_1}^{\beta_1}\cdots b_{r_k,s_k}^{\beta_k})^*(b_{i_1,j_1}^{\alpha_1}\cdots b_{i_l,j_l}^{\alpha_l}))
\prod_{t=1}^k\delta_{\gamma_t \beta_t}\delta_{R_t r_t}\delta_{S_t s_t}\\
 &=&\psi((b_{R_1,S_1}^{\gamma_1}\cdots b_{R_k,S_k}^{\gamma_k})^*(b_{i_1,j_1}^{\alpha_1}\cdots b_{i_l,j_l}^{\alpha_l}))\\
 &=& \delta_{qp}^{\alpha,\gamma}(ij,RS)
\end{eqnarray*}
A second possible case is the composition of $T_p$ with $T_q=\id_B^{\otimes u}\otimes \eta^* \otimes \id_B^{\otimes v}$. Here there are two different situations which deserve to be considered. If $T_p=\eta$, a simple computation shows that $\eta^*\eta=\id_{\C}$ (because $\psi$ is normalized) and the relation (\ref{coeffcomp}) is satisfied. For all the other possible $T_p$, the general diagram is the following:

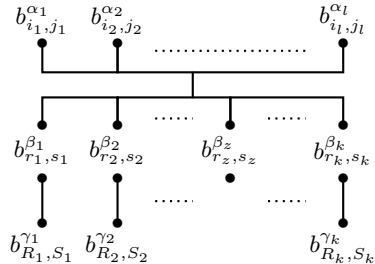
{\centering
 \begin{tikzpicture}[thick,font=\small]
    \path 
    		  (1,1.5) node{$\bullet$} node[above](a){$b_{i_1,j_1}^{\alpha_1}$}
		  (2,1.5) node{$\bullet$} node[above](b){$b_{i_2,j_2}^{\alpha_2}$}
		  (5,1.5) node{$\bullet$} node[above](c){$b_{i_l,j_l}^{\alpha_l}$}          
          
          (1,.4) node{$\bullet$} node[below](d){$b_{r_1,s_1}^{\beta_1}$}
          (2,.4) node{$\bullet$} node[below](e){$b_{r_2,s_2}^{\beta_2}$}
          (3.5,.4) node{$\bullet$} node[below](l){$b_{r_z,s_z}^{\beta_z}$}
          (5,.4) node{$\bullet$} node[below](f){$b_{r_k,s_k}^{\beta_k}$}

          (1,-.3) node{$\bullet$} node[above](dd){}
          (2,-.3) node{$\bullet$} node[above](ee){}
          (3.5,-.3) node{$\bullet$} node[above](ll){}
          (5,-.3) node{$\bullet$} node[above](ff){}
          
          (1,-.9) node{$\bullet$} node[below](g){$b_{R_1,S_1}^{\gamma_1}$}
          (2,-.9) node{$\bullet$} node[below](h){$b_{R_2,S_2}^{\gamma_2}$}
          (5,-.9) node{$\bullet$} node[below](i){$b_{R_k,S_k}^{\gamma_k}$};

    \draw (a) -- +(0,-0.7) -| (b);
    \draw (b) -- +(0,-0.7) -| (c);
    \draw (d) -- +(0,+0.7) -| (e);
    \draw (e) -- +(0,+0.7) -| (l);
    \draw (l) -- +(0,+0.7) -| (f); 
    \draw (dd) -- (g);
    \draw (ee) -- (h);
    \draw (ff) -- (i);
    \draw (3,1.12) -- (3,0.77);
     \draw [dotted] (2.5,1.4) -- (4.5,1.4);
     \draw [dotted] (2.5,.5) -- (3,.5);
     \draw [dotted] (4,.5) -- (4.5,.5);
     \draw [dotted] (2.5,-.6) -- (3,-.6);
     \draw [dotted] (4,-.6) -- (4.5,-.6);
     
\end{tikzpicture}
\par}
\noindent
With respect to the previous case in the lower noncrossing partition $q$, one of the identity maps was replaced by $\eta^*$ (in this case, with a little abuse of notation, we removed $b_{R_z,S_z}^{\gamma_z}$ but we did not reassign the index $z$). 
The relation (\ref{coeffcomp}) is verified because
\begin{eqnarray*}
\lefteqn{\sum_\beta\sum_{r,s=1}^{n_\beta}\delta_p^{\alpha,\beta}(ij,rs)\delta_q^{\beta,\gamma}(rs,RS)}\\
&& \hspace{3cm} =\sum_\beta\sum_{r,s=1}^{n_\beta}\psi((b_{r_1,s_1}^{\beta_1}\cdots b_{r_k,s_k}^{\beta_k})^*(b_{i_1,j_1}^{\alpha_1}\cdots b_{i_l,j_l}^{\alpha_l}))
\psi(b_{r_z,s_z}^{\beta_z})\prod_{t=1,t\neq z}^k\psi((b_{R_t,S_t}^{\gamma_t})^*b_{r_t,s_t}^{\beta_t})\\
&& \hspace{3cm} =\sum_\beta\sum_{r,s=1}^{n_\beta}\psi((b_{r_1,s_1}^{\beta_1}\cdots b_{r_k,s_k}^{\beta_k})^*(b_{i_1,j_1}^{\alpha_1}\cdots b_{i_l,j_l}^{\alpha_l}))
\delta_{r_z s_z}Q_{s_z,\beta_z}^{\um}\prod_{t=1,t\neq z}^k \delta_{\gamma_t \beta_t}\delta_{R_t r_t}\delta_{S_t s_t}\\
&& \hspace{3cm} =\sum_\beta\sum_{r_z=1}^{n_\beta}\psi((b_{R_1,S_1}^{\gamma_1}\cdots e_{r_z,s_z}^{\beta_z}\cdots b_{R_k,S_k}^{\gamma_k})^*(b_{i_1,j_1}^{\alpha_1}\cdots b_{i_l,j_l}^{\alpha_l}))\\
&& \hspace{3cm} =\psi((b_{R_1,S_1}^{\gamma_1}\cdots b_{R_{z-1},S_{z-1}}^{\gamma_{z-1}}b_{R_{z+1},S_{z+1}}^{\gamma_{z+1}}\cdots b_{R_k,S_k}^{\gamma_k})^*(b_{i_1,j_1}^{\alpha_1}\cdots b_{i_l,j_l}^{\alpha_l}))\\
&& \hspace{3cm} =\delta_{qp}^{\alpha,\gamma}(ij,RS)
\end{eqnarray*}

\noindent
The third case we analyse is the composition of $T_p$ with $T_q=\id_B^{\otimes u}\otimes m \otimes \id_B^{\otimes v}$. There are two different situations which deserve to be considered: the two upper points of $m$ can be connected either to one block or to two different blocks of the noncrossing partition $p$. In the first sub-case a cycle appears and the diagram is:

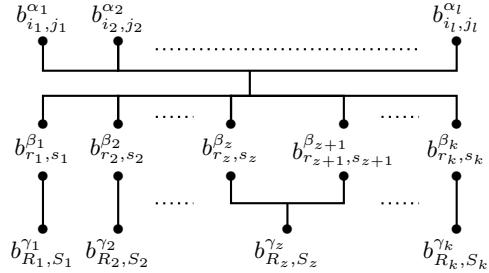
{\centering
 \begin{tikzpicture}[thick,font=\small]
    \path 
    		  (1,1.5) node{$\bullet$} node[above](a){$b_{i_1,j_1}^{\alpha_1}$}
		  (2,1.5) node{$\bullet$} node[above](b){$b_{i_2,j_2}^{\alpha_2}$}
		  (6.5,1.5) node{$\bullet$} node[above](c){$b_{i_l,j_l}^{\alpha_l}$}          
          
          (1,.4) node{$\bullet$} node[below](d){$b_{r_1,s_1}^{\beta_1}$}
          (2,.4) node{$\bullet$} node[below](e){$b_{r_2,s_2}^{\beta_2}$}
          (3.5,.4) node{$\bullet$} node[below](l){$b_{r_z,s_z}^{\beta_z}$}
          (5,.4) node{$\bullet$} node[below](lk){$b_{r_{z+1},s_{z+1}}^{\beta_{z+1}}$}
          (6.5,.4) node{$\bullet$} node[below](f){$b_{r_k,s_k}^{\beta_k}$}

          (1,-.3) node{$\bullet$} node[above](dd){}
          (2,-.3) node{$\bullet$} node[above](ee){}
          (3.5,-.3) node{$\bullet$} node[above](ll){}
          (5,-.3) node{$\bullet$} node[above](llk){}
          (6.5,-.3) node{$\bullet$} node[above](ff){}
          
          (1,-1) node{$\bullet$} node[below](g){$b_{R_1,S_1}^{\gamma_1}$}
          (2,-1) node{$\bullet$} node[below](h){$b_{R_2,S_2}^{\gamma_2}$}
          (4.25,-1) node{$\bullet$} node[below](hk){$b_{R_z,S_z}^{\gamma_z}$}
          (6.5,-1) node{$\bullet$} node[below](i){$b_{R_k,S_k}^{\gamma_k}$};

    \draw (a) -- +(0,-0.7) -| (b);
    \draw (b) -- +(0,-0.7) -| (c);
    \draw (d) -- +(0,+0.7) -| (e);
    \draw (e) -- +(0,+0.7) -| (l);
    \draw (l) -- +(0,+0.7) -| (f);
    \draw (lk) -- +(0,+0.73) -| (l);
    \draw (ll) -- +(0,-0.48) -| (llk);
    \draw (dd) -- (g);
    \draw (ee) -- (h);
    \draw (ff) -- (i);
    \draw (3.75,1.12) -- (3.75,0.77);
    \draw (4.25,-.65) -- (4.25,-1);
     \draw [dotted] (2.5,1.4) -- (6,1.4);
     \draw [dotted] (2.5,.5) -- (3,.5);
     \draw [dotted] (5.5,.5) -- (6,.5);
     \draw [dotted] (2.5,-.65) -- (3,-.65);
     \draw [dotted] (5.5,-.65) -- (6,-.65);
     
\end{tikzpicture}
\par}
\noindent
As in the previous case, one of the points of the lower line was removed and its index (in this case $z+1$) was not reassigned.
The relation (\ref{coeffcomp}) is still verified and the factor $\delta$ implied by the cycle appears. Indeed, we have

\begin{align*}
&\sum_\beta\sum_{r,s=1}^{n_\beta}\delta_p^{\alpha,\beta}(ij,rs)\delta_q^{\beta,\gamma}(rs,RS)\\
& =\sum_\beta\sum_{r,s=1}^{n_\beta}\psi((b_{r_1,s_1}^{\beta_1}\cdots  b_{r_k,s_k}^{\beta_k})^*(b_{i_1,j_1}^{\alpha_1}\cdots b_{i_l,j_l}^{\alpha_l}))
\psi((b_{R_z,S_z}^{\gamma_z})^*b_{r_z,s_z}^{\beta_z}b_{r_{z+1},s_{z+1}}^{\beta_{z+1}})\prod_{t=1,t\neq z,z+1}^k\psi((b_{R_t,S_t}^{\gamma_t})^*b_{r_t,s_t}^{\beta_t})\\
& =\sum_\beta\sum_{r,s=1}^{n_\beta}\psi((b_{r_1,s_1}^{\beta_1}\cdots  b_{r_k,s_k}^{\beta_k})^*(b_{i_1,j_1}^{\alpha_1}\cdots b_{i_l,j_l}^{\alpha_l}))
\delta_{\beta_z\gamma_z}\delta_{\beta_{z+1}\gamma_z}\delta_{r_z R_z} \delta_{s_z r_{z+1}}\delta_{s_{z+1}S_z} Q_{s_z,\beta_z}^{-\um}\prod_{t=1,t\neq z,z+1}^k(\delta_{\gamma_t \beta_t}\delta_{R_t r_t}\delta_{S_t s_t})\\
& =\sum_{s_z=1}^{n_{\gamma_z}}Q_{s_z,\gamma_z}^{-1}\psi((b_{R_1,S_1}^{\gamma_1}\cdots  e_{R_z,s_z}^{\gamma_z}b_{s_z,S_z}^{\gamma_z}b_{R_{z+2},S_{z+2}}^{\gamma_{z+2}} \cdots b_{R_k,S_k}^{\gamma_k})^*(b_{i_1,j_1}^{\alpha_1}\cdots  b_{i_l,j_l}^{\alpha_l}))
\end{align*}
\begin{align*}
&\hspace{-7.3cm} =\delta\cdot\psi((b_{R_1,S_1}^{\gamma_1}\cdots b_{R_{z},S_{z}}^{\gamma_{z}} b_{R_{z+2},S_{z+2}}^{\gamma_{z+2}}\cdots b_{R_k,S_k}^{\gamma_k})^* (b_{i_1,j_1}^{\alpha_1}\cdots b_{i_l,j_l}^{\alpha_l}))\\
&\hspace{-7.3cm} =\delta\cdot\delta_{qp}^{\alpha,\gamma}(ij,RS)
\end{align*}
The diagram of the sub-case with two blocks (with $p\in NC(l+l',k+k')$) follows:

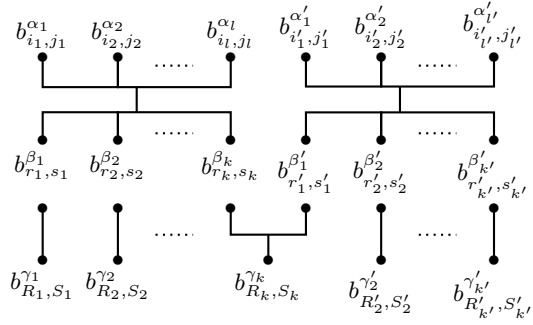
{\centering
 \begin{tikzpicture}[thick,font=\small]
    \path (1,1.5) node{$\bullet$} node[above](a){$b_{i_1,j_1}^{\alpha_1}$}
		  (2,1.5) node{$\bullet$} node[above](b){$b_{i_2,j_2}^{\alpha_2}$}
		  (3.5,1.5) node{$\bullet$} node[above](c){$b_{i_l,j_l}^{\alpha_l}$}          
          
          (1,.4) node{$\bullet$} node[below](d){$b_{r_1,s_1}^{\beta_1}$}
          (2,.4) node{$\bullet$} node[below](e){$b_{r_2,s_2}^{\beta_2}$}
          (3.5,.4) node{$\bullet$} node[below](f){$b_{r_k,s_k}^{\beta_k}$}

    		  (4.5,1.5) node{$\bullet$} node[above](aq){$b_{i_1',j_1'}^{\alpha_1'}$}
		  (5.5,1.5) node{$\bullet$} node[above](bq){$b_{i_2',j_2'}^{\alpha_2'}$}
		  (7,1.5) node{$\bullet$} node[above](cq){$b_{i_{l'}',j_{l'}'}^{\alpha_{l'}'}$}          
          
          (4.5,.4) node{$\bullet$} node[below](dq){$b_{r_1',s_1'}^{\beta_1'}$}
          (5.5,.4) node{$\bullet$} node[below](eq){$b_{r_2',s_2'}^{\beta_2'}$}
          (7,.4) node{$\bullet$} node[below](fq){$b_{r_{k'}',s_{k'}'}^{\beta_{k'}'}$}

          (1,-.5) node{$\bullet$} node[above](dd){}
          (2,-.5) node{$\bullet$} node[above](ee){}
          (3.5,-.5) node{$\bullet$} node[above](ff){}
          
          (1,-1.2) node{$\bullet$} node[below](g){$b_{R_1,S_1}^{\gamma_1}$}
          (2,-1.2) node{$\bullet$} node[below](h){$b_{R_2,S_2}^{\gamma_2}$}
          (4,-1.2) node{$\bullet$} node[below](i){$b_{R_k,S_k}^{\gamma_k}$}
          
          (4.5,-.5) node{$\bullet$} node[above](ddq){}
          (5.5,-.5) node{$\bullet$} node[above](eeq){}
          (7,-.5) node{$\bullet$} node[above](ffq){}
          
          (5.5,-1.2) node{$\bullet$} node[below](hq){$b_{R_2',S_2'}^{\gamma_2'}$}
          (7,-1.2) node{$\bullet$} node[below](iq){$b_{R_{k'}',S_{k'}'}^{\gamma_{k'}'}$};

    \draw (a) -- +(0,-0.7) -| (b);
    \draw (b) -- +(0,-0.7) -| (c);
    \draw (d) -- +(0,+0.7) -| (e);
    \draw (e) -- +(0,+0.7) -| (f);   
    \draw (aq) -- +(0,-0.79) -| (bq);
    \draw (bq) -- +(0,-0.79) -| (cq);
    \draw (dq) -- +(0,+0.79) -| (eq);
    \draw (eq) -- +(0,+0.79) -| (fq);
    \draw (ff) -- +(0,-0.48) -| (ddq);
    \draw (dd) -- (g);
    \draw (ee) -- (h);
    \draw (eeq) -- (hq);
    \draw (ffq) -- (iq);
    \draw (2.25,1.12) -- (2.25,0.77);
    \draw (5.75,1.12) -- (5.75,0.77);
    \draw (4,-.85) -- (4,-1.2);
     \draw [dotted] (6,1.4) -- (6.5,1.4);
     \draw [dotted] (6,.5) -- (6.5,.5);
     \draw [dotted] (6,-.85) -- (6.5,-.85); 
     \draw [dotted] (2.5,1.4) -- (3,1.4);
     \draw [dotted] (2.5,.5) -- (3,.5);
     \draw [dotted] (2.5,-.85) -- (3,-.85);
     
\end{tikzpicture}
\par}
\noindent
Relation (\ref{coeffcomp}) is still verified, indeed
\begin{align*}
&\sum_{\beta,\beta'}\sum_{r,s=1}^{n_\beta}\sum_{r',s'=1}^{n_{\beta'}}\delta_p^{\alpha\alpha',\beta\beta'}(ii'jj,'rr'ss')\delta_q^{\beta\beta',\gamma\gamma'}(rr'ss',RR'SS')=\\
&\sum_{\beta,\beta'}\sum_{r,s=1}^{n_\beta}\sum_{r',s'=1}^{n_{\beta'}}\psi((b_{r_1,s_1}^{\beta_1}\cdots  b_{r_k,s_k}^{\beta_k})^*(b_{i_1,j_1}^{\alpha_1}\cdots b_{i_l,j_l}^{\alpha_l})) \psi((b_{r_1',s_1'}^{\beta_1'}\cdots  b_{r_{k'}',s_{k'}'}^{\beta_{k'}'})^*(b_{i_1',j_1'}^{\alpha_1'}\cdots  b_{i_{l'}',j_{l'}'}^{\alpha_l}))\\
&\hspace{5cm}\prod_{t=1}^{k-1}\psi((b_{R_t,S_t}^{\gamma_t})^*b_{r_t,s_t}^{\beta_t}) \psi((b_{R_k,S_k}^{\gamma_k})^*b_{r_k,s_k}^{\beta_k}b_{r_1',s_1'}^{\beta_1'}) \prod_{t=2}^{k'}\psi((b_{R_t',S_t'}^{\gamma_t'})^*b_{r_t',s_t'}^{\beta_t'})=\\
&\sum_{\beta,\beta'}\sum_{r,s=1}^{n_\beta}\sum_{r',s'=1}^{n_{\beta'}}\psi((b_{r_1,s_1}^{\beta_1}\cdots  b_{r_k,s_k}^{\beta_k})^*(b_{i_1,j_1}^{\alpha_1}\cdots b_{i_l,j_l}^{\alpha_l})) \psi((b_{r_1',s_1'}^{\beta_1'}\cdots  b_{r_{k'}',s_{k'}'}^{\beta_{k'}'})^*(b_{i_1',j_1'}^{\alpha_1'}\cdots  b_{i_{l'}',j_{l'}'}^{\alpha_l}))\\
&\hspace{5cm}\prod_{t=1}^{k-1}(\delta_{\gamma_t \beta_t}\delta_{R_t r_t}\delta_{S_t s_t}) \delta_{\beta_k\gamma_k}\delta_{\beta_1'\gamma_k}\delta_{r_k R_k} \delta_{s_k r_1'}\delta_{s_1'S_k} Q_{s_k,\beta_k}^{-\um} \prod_{t=2}^{k'}(\delta_{\gamma_t' \beta_t'}\delta_{R_t' r_t'}\delta_{S_t' s_t'})=\\
&\sum_{s_k=1}^{n_{\gamma_k}}Q_{s_k,\gamma_k}^{-\um}\psi((b_{R_1,S_1}^{\gamma_1} \cdots b_{R_k,s_k}^{\gamma_k})^*(b_{i_1,j_1}^{\alpha_1}\cdots  b_{i_l,j_l}^{\alpha_l})) \psi((b_{s_k,S_k}^{\gamma_k}b_{R_2',S_2'}^{\gamma_2'} \cdots b_{R_{k'}',S_{k'}}^{\gamma_{k'}'})^*(b_{i_1',j_1'}^{\alpha_1'}\cdots  b_{i_{l'}',j_{l'}'}^{\alpha_{l'}'}))=\\
& Q_{j_l,\gamma_k}^{-\um}\psi((b_{R_1,S_1}^{\gamma_1} \cdots b_{R_k,j_l}^{\gamma_k})^*(b_{i_1,j_1}^{\alpha_1}\cdots  b_{i_l,j_l}^{\alpha_l})) \psi((b_{j_l,S_k}^{\gamma_k}b_{R_2',S_2'}^{\gamma_2'} \cdots b_{R_{k'}',S_{k'}'}^{\gamma_{k'}'})^*(b_{i_1',j_1'}^{\alpha_1'}\cdots  b_{i_{l'}',j_{l'}'}^{\alpha_{l'}'}))=\\
& Q_{j_l,\gamma_k}^{-1}\psi((b_{R_1,S_1}^{\gamma_1} \cdots e_{R_k,j_l}^{\gamma_k})^*(b_{i_1,j_1}^{\alpha_1}\cdots  b_{i_l,j_l}^{\alpha_l})) \psi((b_{j_l,S_k}^{\gamma_k}b_{R_2',S_2'}^{\gamma_2'} \cdots b_{R_{k'}',S_{k'}'}^{\gamma_{k'}'})^*(b_{i_1',j_1'}^{\alpha_1'}\cdots  b_{i_{l'}',j_{l'}'}^{\alpha_{l'}'}))=\\
& \psi((b_{R_1,S_1}^{\gamma_1}\cdots b_{R_k,S_k}^{\gamma_k}b_{R_2',S_2'}^{\gamma_2'} \cdots b_{R_{k'}',S_{k'}'}^{\gamma_{k'}'})^*(b_{i_1,j_1}^{\alpha_1}\cdots  b_{i_l,j_l}^{\alpha_l}b_{i_1',j_1'}^{\alpha_1'}\cdots  b_{i_{l'}',j_{l'}'}^{\alpha_{l'}'}))=\\
& \delta_{qp}^{\alpha\alpha',\gamma\gamma'}(ii'jj',RR'SS')
\end{align*}

\noindent
where the second to last equality is obtained by simplifying $Q_{j_l,\gamma_k}^{-1}$ with the value given by the first $\psi$ and by inserting all the $\delta$ conditions and coefficients of its argument in the argument of the second $\psi$. This is essentially due to the fact that the index $j_l$ is in both arguments.

With similar computations, it is finally possible to prove that the formula still holds in the remaining cases of $\eta$ (trivial case) and $m^*$.
\end{proof}

\begin{remark}
On the one hand, in the classical case of $S_n^+=\G^{aut}(\C^n, \tr)$ (see Proposition 1.9 in \cite{bs09}), the composition formula we have just proved contains the scalar coefficient $n^{-cb(p,q)}$, which in this case disappears because $\psi$ is normalized (while $\tr(1)=n$). On the other hand, we need to introduce the correction factor $\delta^{-cy(p,q)}$ which is necessary because $\psi$ is a $\delta$-form (while $\tr$ is a 1-form).
\end{remark}

\begin{remark}
From Proposition \ref{propmap} it follows that the category $\mathscr{NC}$ of the noncrossing partitions, with objects the elements of $\N$ and morphisms $\Hom(k,l)=span\{T_p|p\in NC(k,l)\}$, is a rigid concrete monoidal C*-category.
Indeed, it is endowed with a canonical fiber functor $\mathscr{NC}\longrightarrow \hilb$ which sends every $n\in Ob(\mathscr{NC})$ to the Hilbert space $B^{\ot n}$ and every object is equal to its conjugate. 
For more details on these notions we refer to \cite{nt13,wor88}.
\end{remark}

In conclusion, by generalizing a result from \cite{bs09}, we can describe a basis of the space of intertwiners of $\aut$ by making use of noncrossing partitions (instead of Temperley-Lieb diagrams as in \cite{ban02}).

\begin{theorem}\label{intertwinersgaut}
Let $B$ be a $n$-dimensional C*-algebra, $n\geq 4$ and consider the compact quantum automorphism group $\aut$ with fundamental representation $u$. Then, for all $k,l\in\N$
$$\Hom(u^{\otimes k},u^{\otimes l})=span\{T_p | p\in NC(k,l)\}$$
Furthermore, the maps associated to distinct noncrossing partitions in $NC(k,l)$ are linearly independent.
\end{theorem}

\begin{proof}
For the first inclusion ($\supseteq$) it is enough to observe that all noncrossing partitions can be obtained from the basic ones (diagrams of multiplication, unity and identity) by using the operations of Definition \ref{opdiag} (this is true because the theorem has already been proved for $(B,\psi)=(\C^n,\tr)$). The inclusion follows because the maps associated to these basic diagrams are intertwiners.\\
For the second inclusion ($\subseteq$) we apply the Tannaka-Krein duality (see \cite{wor88}) to the concrete rigid monoidal C*-category $\mathscr{NC}$. This implies that there exists a compact quantum group $\G=(C(\G),\Delta)$ with fundamental representation $v$ and such that $\Hom(v^{\ot k},v^{\ot l})=span\{T_p | p\in NC(k,l)\}$. Because of the universality of the Tannaka-Krein construction, from the inclusion already proved it follows that there is a surjective map $\phi:C(\G)\longrightarrow \aut$ such that $(id_B\ot\phi)(v)=u$. In order to complete the proof we have to show that the map is an isomorphism. This follows from the universality of the quantum automorphism group construction after observing that the matrix $v$ is unitary and verifying the two conditions $m\in\Hom(v^{\ot 2},v)$ and $\eta\in\Hom(1,v)$ because $m$ and $\eta$ correspond to two noncrossing partitions.\\
The independence of the maps follows from a dimension count, as observed in \cite{bs09}, because the dimensions of the intertwining spaces computed in \cite{ban99} are still true in this case (see also \cite{ban02}).
\end{proof}

We recall now a proposition attributed to Brannan in \cite{dcfy} in order to generalize the representation theory to the case of a state $\psi$.

\begin{prop}\label{freprod}
Let $(B,\psi)$ be a finite-dimensional C*-algebra equipped with a state. 
Let $B=\bigoplus_{i=1}^k B_i$ be the coarsest direct sum decomposition into C*-algebras such that, for each $i$, the normalization $\psi_i$ of $\psi_{|_{B_i}}$ is a $\delta_i$-form for a suitable $\delta_i$. Then, $\aut$ is isomorphic to the free product $\hat{\ast}_{i=1}^k \G^{aut}(B_i,\psi_i)$.
\end{prop}

The representation theory of a free product has been described by Wang in \cite{wan95} and it is completely determined by the representation theory of the factors.

\section{Intertwiners and fusion rules of $\cwpr$}

In this section, we introduce the free wreath product of a discrete group by a quantum automorphism group and we consider its representation theory.
 
\begin{deff}\label{wrpr}
Let $\Gamma$ be a discrete group and consider the quantum automorphism group $\aut$ where $\psi$ is a faithful state on $B$. The free wreath product $\ccwpr$ is the universal unital C*-algebra generated by the coefficients of $a(g)\in\mathcal{L}(B)\otimes \ccwpr$, $g\in\Gamma$ and with relations such that:
\begin{itemize}
\item $a(g)$ is unitary for every $g\in\Gamma$,
\item $m\in \Hom(a(g)\otimes a(h),a(gh))$ for every $g,h\in\Gamma$,
\item $\eta\in \Hom(1,a(e))$.
\end{itemize}
\end{deff}

Such a universal C*-algebra can be endowed with a compact quantum group structure, but as far as this construction is concerned, we need to go deeper into the generators $a(g)$.

\begin{notaz}\label{prodmat}
Consider the matrices $a=(a_{ij,\alpha}^{kl,\beta})$ and $b=(b_{ij,\alpha}^{kl,\beta})$ with coefficients in a C*-algebra where $1\leq \alpha$, $\beta\leq c$, $1\leq i,j\leq n_\alpha$, $1\leq k,l\leq n_\beta$.
The operations of multiplication and adjoint are defined respectively by: $$(ab)^{kl,\beta}_{ij,\alpha}=\sum_{\gamma=1}^c\sum_{r,s=1}^{n_\gamma}a^{kl,\beta}_{rs,\gamma}b^{rs,\gamma}_{ij,\alpha}\qquad\qquad a^*=((a_{kl,\beta}^{ij,\alpha})^*)_{ij,\alpha}^{kl,\beta}$$
\end{notaz}

As in the case of $\caut$, the definition of the universal C*-algebra $\ccwpr$ does not depend on the choice of an orthonormal basis of $B$.
When it will be necessary to fix a basis of $B$, we will always use $\mathscr{B}'$, as this will allow us to consider as diagonal the matrices $Q_\alpha$ associated to the state $\psi$.

\begin{remark}\label{relwpr}
Let us fix $\mathscr{B}'$ as basis of the C*-algebra $B$. Then, the generators of the C*-algebra $\ccwpr$ can be seen as matrices of type $a(g)=(a_{ij,\alpha}^{kl,\beta}(g))$, $1\leq \alpha,\beta\leq c$, $1\leq i,j\leq n_\alpha$, $1\leq k,l\leq n_\beta$, $g\in\Gamma$.
By using the conventions introduced in Notation \ref{prodmat}, we can change the three conditions of Definition \ref{wrpr} into the following relations:
$$\sum_{l=1}^{n_\gamma}Q_{l,\gamma}^{-\um}a_{ik,\alpha}^{rl,\gamma}(g)a_{pj,\beta}^{ls,\gamma}(h)=\delta_{\alpha\beta}\delta_{kp}Q_{k,\alpha}^{-\um}a_{ij,\alpha}^{rs,\gamma}(gh)$$
$$\sum_{k=1}^{n_\alpha}Q_{k,\alpha}^{-\um}a_{ik,\alpha}^{rp,\beta}(g)a_{kj,\alpha}^{qs,\gamma}(h)=\delta_{\beta\gamma}\delta_{pq}Q_{p,\beta}^{-\um}a_{ij,\alpha}^{rs,\beta}(gh)
$$
$$\sum_{\alpha=1}^{c}\sum_{j=1}^{n_\alpha}Q_{j,\alpha}^{\um}a_{jj,\alpha}^{kl,\beta}(e)=\delta_{kl}Q_{l,\beta}^{\um}
\qquad \qquad 
\sum_{\beta=1}^{c}\sum_{k=1}^{n_\beta}Q_{k,\beta}^{\um}a_{ij,\alpha}^{kk,\beta}(e)=\delta_{ij}Q_{i,\alpha}^{\um}
$$

\begin{equation}\label{adj}
(a_{ij,\alpha}^{kl,\beta}(g))^*=(\frac{Q_{l,\beta}}{Q_{j,\alpha}})^{\um}(\frac{Q_{k,\beta}}{Q_{i,\alpha}})^{-\um}a_{ji,\alpha}^{lk,\beta}(g^{-1})
\end{equation}
\end{remark}

\begin{prop} 
There exists a unique $\ast$-homomorphism
$$\Delta:\ccwpr\longrightarrow\ccwpr\otimes\ccwpr$$ such that, for any $g\in\Gamma$ $$(\id\ot\Delta)(a(g))=a(g)_{(12)}a(g)_{(13)}$$
Moreover, $\Delta$ is a comultiplication and the pair $(\ccwpr,\Delta)$ is a compact quantum group which is called the free wreath product of $\widehat{\Gamma}$ by $\aut$ and will be denoted $\cwpr$ or $\wpr$.
\end{prop}

\begin{proof}
In order to prove the existence of $\Delta$, we have to check that the images of the generators $a(g)$ satisfy the same relations. It is clear that $a(g)_{(12)}a(g)_{(13)}$ is unitary. When considering the condition on the multiplication, we have\\ 
$\begin{array}{lll}
(m\ot 1^{\ot 2})(a(g)_{(12)}a(g)_{(13)}\ot a(h)_{(12)}a(h)_{(23)})&=& 
(m\ot 1^{\ot 2})(a(g)_{(13)}a(g)_{(14)}a(h)_{(23)}a(h)_{(24)})\\ &=& 
(m\ot 1^{\ot 2})(a(g)_{(13)}a(h)_{(23)})(a(g)_{(14)}a(h)_{(24)})\\ &=&  
(a(gh)_{(12)} a(gh)_{(13)})(m\ot 1^{\ot 2})
\end{array}$\\
The condition on the unity map is simply $a(e)_{(12)}a(e)_{(13)}(\eta\ot 1^{\ot 2})=\eta\ot 1^{\ot 2}$.\\
Therefore, by the universality of the free wreath product construction, the existence of the map $\Delta$ is proved. The uniqueness is an immediate consequence of the fact that the image of all the generators is fixed.\\
Now, we have to verify that the defining properties of a compact quantum group are satisfied. We observe that the matrices $a(g)$ are unitary and, by construction, their entries generate a dense $\ast$-subalgebra of $\ccwpr$. We just proved the existence of a suitable comultiplication $\Delta$. What is left is to prove that the transposed matrices $(a(g))^t$ are invertible. The basis of $B$ which will be used for the computations is $\mathscr{B}'$. 
For every $(a(g))^t$ the inverse is given by $b(g)=((\frac{Q_{l,\beta}}{Q_{j,\alpha}})^{-\um}(\frac{Q_{k,\beta}}{Q_{i,\alpha}})^{\um}a_{ji,\alpha}^{lk,\beta}(g^{-1}))_{ij,\alpha}^{kl,\beta}$. Indeed, we have that
$$\begin{array}{lll}
(b(g)(a(g))^t)_{rs,\gamma}^{kl,\beta}& =&\sum_{\alpha=1}^c\sum_{i,j=1}^{n_\alpha}(\frac{Q_{l,\beta}}{Q_{j,\alpha}})^{-\um}(\frac{Q_{k,\beta}}{Q_{i,\alpha}})^{\um}a_{ji,\alpha}^{lk,\beta}(g^{-1})a(g)_{ij,\alpha}^{rs,\gamma} \\ &=& \delta_{\beta\gamma}\delta_{kr}\sum_{\alpha=1}^c\sum_{j=1}^{n_\alpha}(\frac{Q_{l,\beta}}{Q_{j,\alpha}})^{-\um}a_{jj,\alpha}^{ls,\beta}(e) \\ &=& \delta_{\beta,\gamma}\delta_{kr}\delta_{ls}
\end{array}$$
\noindent
so $b(g)(a(g))^t=\text{Id}$. In the same way it is possible to prove that 
$(a(g))^t b(g)=\text{Id}$. It follows that $a(g)$ is invertible and $\wpr$ is a compact quantum group.
\end{proof}

\noindent
We can now focus on the representation theory of $H_{(B,\psi)}^+(\Gamma)$ when $\psi$ is a $\delta$-form.

\begin{notaz}\label{prodec}
We denote by $NC_{\widehat{\Gamma}}(g_1,\dots ,g_k;h_1,\dots ,h_l)$ the set of diagrams in $NC(k,l)$ where the $k$ upper points are decorated by some $g_i\in\Gamma$ and the $l$ lower points by elements $h_j\in\Gamma$ such that, in every block, the product of the upper elements is equal to the product of the lower elements (with the convention that, if the block connects only upper or only lower points, the product must be the unit of $\Gamma$). For example 

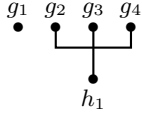
{\centering
 \begin{tikzpicture}[thick,font=\small]
    \path (0,1) node{$\bullet$} node[above](c) {$g_1$}
          (0.5,1) node{$\bullet$} node[above](d) {$g_2$}
		  (1,1) node{$\bullet$} node[above](e) {$g_3$}
          (1.5,1) node{$\bullet$} node[above](f) {$g_4$}
          (1,.3) node{$\bullet$} node[below](h) {$h_1$};
  
    \draw (d) -- +(0,-.5) -| (f);
    \draw (e) -- (h);
 \end{tikzpicture}
\par}
\noindent
is in $NC_{\widehat{\Gamma}}(g_1,g_2,g_3,g_4;h_1)$ if $g_1=e$, $g_2g_3g_4=h_1$.
\end{notaz}
The operations between noncrossing partitions introduced in Definition \ref{opdiag} as well as the results of Proposition \ref{propmap} naturally extend to decorated diagrams.

\begin{prop}
Let $p\in NC_{\widehat{\Gamma}}(g_1,\dots ,g_k;h_1,\dots ,h_l)$, $q\in NC_{\widehat{\Gamma}}(g'_1,\dots ,g'_v;h'_1,\dots ,h'_w)$.
We have:
\begin{enumerate}
\item $T_{p\otimes q}=T_p\otimes T_q$ where $p\otimes q\in NC_{\widehat{\Gamma}}(g_1,\dots ,g_k,g'_1,\dots ,g'_v;h_1,\dots ,h_l,h'_1,\dots ,h'_w)$ is obtained by horizontal concatenation,
\item $T_p^*=T_{p^*}$ where $p^*\in NC_{\widehat{\Gamma}}(h_1,\dots ,h_l;g_1,\dots ,g_k)$ is obtained by reflecting $p$ with respect to an horizontal line which lies between the rows of the upper and of the lower points,
\item if $(h_1,\dots ,h_l)=(g'_1,\dots ,g'_v)$ then $T_{qp}=\delta^{-cy(p,q)}T_q T_p$ where $qp\in NC_{\widehat{\Gamma}}(g_1,\dots ,g_k;h'_1,\dots ,h'_w)$ is obtained by vertical concatenation.
\end{enumerate}
\end{prop}

\begin{proof}
The proof is essentially the same as of Proposition \ref{propmap}, we have only to observe that the operations between noncrossing partitions are well defined with respect to the decoration of the diagrams, i.e. the operations of tensor product, adjoint and composition always produce diagrams with an admissible decoration.
\end{proof}

\begin{example}
The fundamental maps $m$, $\eta$ and $\id_B$ can be represented by using decorated noncrossing partitions. In particular, for all $g,h\in\Gamma$ the multiplication $m$, the unity $\eta$ and the identity $\id_B$ respectively correspond to the following diagrams.

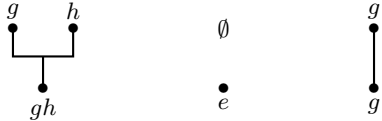
{\centering
 \begin{tikzpicture}[thick,font=\small]
    \path (0,.8) node{$\bullet$} node[above](a) {$g$}
          (.8,.8) node{$\bullet$} node[above](b) {$h$}
          (.4,0) node{$\bullet$} node[below](c) {$gh$}
          
          (2.8,.8)node{$\emptyset$}
          (2.8,0) node{$\bullet$} node[below](d) {$e$}
          
          (4.8,.8) node{$\bullet$} node[above](e) {$g$}
          (4.8,0) node{$\bullet$} node[below](f) {$g$};

    \draw (c) -- +(0,0.7);
    \draw (a) -- +(0,-0.6) -| (b);
    \draw (e) -- (f);
 \end{tikzpicture}
\par}

\end{example}

\begin{theorem}\label{teointertgamma}
Let $\Gamma$ be a discrete group and $(B,\psi)$ be a finite dimensional C*-algebra with a $\delta$-form $\psi$ and $\dim(B)\geq 4$. The spaces of intertwiners of $\widehat{\Gamma}\wr_*\aut$ are spanned by the linear maps associated to some decorated noncrossing partitions. In particular for any $g_i,h_j\in\Gamma$ we have
$$Hom(\bigotimes_{i=1}^k a(g_i),\bigotimes_{j=1}^l a(h_j))= span\{T_p|p\in NC_{\widehat{\Gamma}}(g_1,\dots,g_k;h_1,\dots,h_l)\}$$
with the convention that, if $k=0$, $\bigotimes_{i=1}^k a(g_i)=1_{\wpr}$ and the space of the noncrossing partitions is $NC_{\widehat{\Gamma}}(\emptyset;h_1,\dots,h_l)$, i.e. it does not have upper points. Similarly, if $l=0$.
\end{theorem}

\begin{proof}
We prove this result by showing the double inclusion. The first inclusion we take into account is the one of the right space in the left one ($\supseteq$).
It is well known that all noncrossing partitions can be built by using the operations of tensor product, composition and adjoint on the noncrossing partitions corresponding to the maps of multiplication, unity and identity. This fact can be easily generalized to the context of the noncrossing partitions decorated with the elements of $\Gamma$. Let $p\in NC_{\widehat{\Gamma}}(g_1,...,g_k;h_1,...,h_l)$ be a decorated noncrossing partition. Its decomposition in terms of 
the decorated noncrossing partitions corresponding to $m$, $\eta$ and $\id$ is simply obtained by considering the usual decomposition in terms of (non decorated) noncrossing partition and by observing that it is always possible to decorate all these partitions in an admissible way.
Now, the linear maps corresponding to the decorated noncrossing partitions of the decomposition are intertwiners of $\wpr$ by definition of free wreath product. It follows that $T_p\in\Hom(\bigotimes_{i=1}^k a(g_i),\bigotimes_{j=1}^l a(h_j))$.\\
For the second inclusion ($\subseteq$), we observe that, similarly to the proof of Theorem \ref{intertwinersgaut}, the noncrossing partitions decorated with the elements of $\Gamma$ form a concrete rigid monoidal C*-category $\mathscr{NC}_{\widehat{\Gamma}}$, whose objects are the finite sequences $(g_1,...,g_k), g_i\in\Gamma$ and whose spaces of morphisms are $\Hom((g_1,...,g_k),(h_1,...,h_l))=span\{T_p | p\in NC_{\widehat{\Gamma}}(g_1,...,g_k;h_1,...,h_l)\}$. Therefore, by the Tannaka-Krein duality, there exists a compact quantum group $\G=(C(\G),\Delta)$, such that $C(\G)$ is generated by the coefficients of a family of finite dimensional unitary representations $a(g_i)'$ and $\Hom(\bigotimes_{i=1}^k a(g_i)',\bigotimes_{j=1}^l a(h_j)')=span\{T_p | p\in NC_{\widehat{\Gamma}}(g_1,...,g_k;h_1,...,h_l)\}$. Moreover, the inclusion showed in the first  part of the proof, together with the universality of the Tannaka-Krein construction, imply that there is a surjective map $\phi:C(\G)\longrightarrow \wpr$ such that $(\id\ot \phi)(a(g)')=a(g)$, for all $g\in\Gamma$. In order to complete the proof we have to show that the map is an isomorphism. We observe that the representations $a(g)'$ are such that $m\in\Hom(a(g)'\ot a(h)',a(gh)')$ and $\eta\in\Hom(1,a(e)')$ because these maps correspond to well decorated noncrossing partitions. Therefore, because of the universality of the free wreath product construction we have the inverse morphism and the proof is complete.
\end{proof}
\noindent
In the same way as in \cite[Cor 2.21]{fr13}, it is then possible to prove that

\begin{prop}
The basic representations $a(g)$, $g\in\Gamma$ of $\wpr$ are irreducible and pairwise non-equivalent if $g\neq e$; the remaining representation is $a(e)=1\oplus \omega(e)$, where $\omega(e)$ is irreducible and non-equivalent to any $a(g), g\neq e$.
\end{prop}

\begin{deff}
Let $\Gamma$ be a discrete group and $M=<\Gamma>$ the monoid of the words over $\Gamma$. It is endowed with the following operations:
\begin{itemize}
\item[-] involution: $(g_1,\dots ,g_k)^-=(g_k^{-1},\dots ,g_1^{-1})$
\item[-] concatenation: $(g_1,\dots ,g_k),(h_1,\dots ,h_l)=(g_1,\dots ,g_k,h_1,\dots ,h_l)$
\item[-] fusion: $(g_1,\dots ,g_k).(h_1,\dots ,h_l)=(g_1,\dots ,g_kh_1,\dots ,h_l)$
\end{itemize}
\end{deff}

We can now state the main theorem, generalizing Theorem 2.25 in \cite{fr13} to the case of $\wpr$. The proof is the same as of \cite{fr13}, because it relies only on the description of the intertwining spaces by using noncrossing partitions.

\begin{theorem}
The irreducible representations of $\wpr$ are indexed by the words of $M$ and denoted $\omega(x), x\in M$ with involution $\bar{\omega}(x)=\omega(\bar{x})$. In particular, for $g\in\Gamma$, we have $\omega(g)=a(g)\ominus \delta_{g,e}1$.

The fusion rules are:
$$\omega(x)\otimes \omega(y)=\sum_{\substack{x=u,t \\ y=\bar{t},v}}\omega(u,v)\oplus \sum_{\substack{x=u,t \\ y=\bar{t},v \\ u\neq\emptyset , v\neq \emptyset}}\omega(u.v)$$
\end{theorem}

\noindent
In order to extend these results to the case of a state $\psi$, we use this proposition.

\begin{prop}\label{freeprodcqg}
Let $B=\bigoplus_{\alpha=1}^c M_{n_\alpha}(\C)$ be a finite-dimensional C*-algebra with a state $\psi=\bigoplus_{\alpha=1}^c \Tr(Q_\alpha\cdot)$ on it.
The state $\psi$ restricted to every summand $M_{n_\alpha}(\C)$ (and normalized) is a $\delta$-form with $\delta=\Tr(Q_\alpha^{-1})$. 
Consider the decomposition $B=\bigoplus_{i=1}^{d}B_{i}$ where every $B_i$ is the direct sum of all the $M_{n_\alpha}(\C)$ such that $\Tr(Q_\alpha^{-1})$ is a constant value denoted $\delta_i$. Let $\psi_i$ be the state on $B_i$ which is obtained by normalizing $\psi_{|_{B_i}}$. Then
$$\cwpr\cong \hat{\ast}_{i=1}^d (\widehat{\Gamma}\wr_{*}\G^{aut}(B_i,\psi_i))$$
is a $*$-isomorphism which intertwines the comultiplications.
\end{prop}

\begin{proof}
For this proof we will use the relations introduced in Remark \ref{relwpr} between the generators of the free wreath product $\ccwpr$. Firstly, we notice that the map $mm^*\in \Hom(a(g),a(g))$ is such that $mm^*_{|_{B_i}}=\delta_i \id_{B_i}$. The existence of this particular morphism implies that $a_{ij,\alpha}^{kl,\beta}(g)=0$ for all $i,j,k,l$ whenever $M_{n_\alpha}(\C)$ and $M_{n_\beta}(\C)$ are not in the same summand $B_i$. This allows us to simplify the defining relations above and we deduce that, for every $i$, there is a natural morphism $C(\widehat{\Gamma}\wr_{*}\G^{aut}(B_i,\psi_i))\longrightarrow \ccwpr$. It follows that, because of the universality of the free product construction, there is a morphism which intertwines the comultiplications $$\ast_{i=1}^k C(\widehat{\Gamma}\wr_{*}\G^{aut}(B_i,\psi_i))\longrightarrow \ccwpr$$
Conversely, by joining the generators and the relations of the different factors of $\ast_{i=1}^k C(\widehat{\Gamma}\wr_{*}\G^{aut}(B_i,\psi_i))$, we exactly get the generators and the relations of $\wpr$ (in the simplified version). Because of the universality of the free wreath product construction we get the inverse morphism which intertwines the comultiplications
$$\ccwpr\longrightarrow \ast_{i=1}^k C(\widehat{\Gamma}\wr_{*}\G^{aut}(B_i,\psi_i))$$
and the quantum group isomorphism is proved.
\end{proof}

The non trivial irreducible representations of $\cwpr$ are then given by an alternating tensor product of non-trivial irreducible representations of the factors $\widehat{\Gamma}\wr_{*}\G^{aut}(B_i,\psi_i)$ (see \cite{wan95}).

We conclude this section with a remark about the spectral measure on a subalgebra of $\wpr$ ($\psi$ $\delta$-trace) which will be useful in the following section.

\begin{remark}\label{carae}
It is well known that the character of the fundamental representation of a quantum permutation group follows the free Poisson law (see e.g. \cite{bc07}). Let $\chi(a(e)):=(\Tr\otimes \id)(a(e))$ be the character of the representation $a(e)$ of $\wpr$ ($\psi$ is a $\delta$-form). The character $\chi(a(e))$ is self-adjoint because of the relation (\ref{adj}). Therefore, in order to find its spectral measure, it is enough to compute the moments $h(\chi(a(e))^k)$. By denoting $p_k$ the orthogonal projection onto the fixed points space $\Hom(1,a(e)^{\otimes k})$ and by using some classic results of Woronowicz (see \cite{wor88}) we have 
$h(\chi(a(e))^k)=h((\Tr\otimes \id)(a(e))^k)=\Tr((\id\otimes h)(a(e)^{\otimes k}))=\Tr(p_k)=\dim(\Hom(1,a(e)^{\otimes k}))=\text{\#}NC(0,k)=C_k$ where $C_k$ are the Catalan numbers. They are exactly the moments of the free Poisson law of parameter 1 so this is the spectral measure of $\chi(a(e))$.
\end{remark}

\section{Properties of $\wpr$}

\subsection{Simplicity and uniqueness of the trace for the reduced algebra}

We prove that, under certain conditions, $C_r(\wpr)$ is simple with unique trace.

\begin{remark}\label{decomp}
The free product decomposition given in Proposition \ref{freeprodcqg} implies that the Haar measure of $\cwpr$ is the free product of the Haar measures of its factors, by using a well known result of Wang (see \cite{wan95}). 
It follows that the decomposition is still true at the level of the reduced C* and von Neumann algebras so the following isomorphisms hold:
$$(C_r(\cwpr),h)\cong \underset{\text{red}}{\ast}_{i=1}^k (C_r(\widehat{\Gamma}\wr_{*}\G^{aut}(B_i,\psi_i)),h_i)$$
$$(L^\infty(\cwpr),h)\cong \ast_{i=1}^k (L^\infty(\widehat{\Gamma}\wr_{*}\G^{aut}(B_i,\psi_i)),h_i)$$
where $h$ and $h_i$ are the Haar states on the respective C*-algebras.
\end{remark}

\begin{prop}\label{semplicita}
Let $(B,\psi)$ be a finite dimensional C*-algebra endowed with a trace $\psi$. Let $\Gamma$ be a discrete group, $|\Gamma |\geq 4$. Consider the free product decomposition of the reduced C*-algebra $C_r(\wpr)$ given in Remark \ref{decomp}. If there is either only one factor (i.e. $\psi$ is a $\delta$-trace) and $\dim(B)\geq 8$ or there are two or more factors with $\dim(B_i)\geq 4$ for all $i$, then $C_r(\wpr)$ is simple with a unique trace given by the free product of the Haar measures.
\end{prop}

\begin{proof}
If there are two or more factors the result follows from a proposition of Avitzour (see \cite[Section 3]{avi82}). It states that, given two C*-algebras $A$ and $B$ endowed with tracial Haar states $h_A$ and $h_B$, the reduced free product C*-algebra $A\underset{\text{red}}{\ast} B$ is simple with unique trace, if there exist two unitary elements of $\text{Ker}(h_A)$ which are orthogonal with respect to the scalar product induced by $h_A$ and a unitary element in $\text{Ker}(h_B)$. In order to show that, in our case, these elements exist we use a result from \cite[Proposition 4.1 (i)]{dhr97} according to which, if a C*-algebra $A$ endowed with a normalized trace $\tau$ admits an abelian sub-C*-algebra $B$ so that the spectral measure corresponding to $\tau_{|_B}$ is diffuse, then there is a unitary element $u\in A$ such that $\tau(u^n)=0$ for each $n\in\Z, n\neq 0$.
We aim to apply such a proposition to every factor of the decomposition in order to satisfy the Avitzour's condition. But this follows from Remark \ref{carae} where we observed that, when considering the generator $a(e)$ of an indecomposable free wreath product $\cwpr$, $\psi$ $\delta$-trace, the spectral measure associated to its character $\chi(a(e))$ is the free Poisson law of parameter 1, which is diffuse. The simplicity and uniqueness of the trace in the multifactor case are then proved.

In the second case, when $\psi$ is a $\delta$-trace and there is not a free product decomposition, the proof is a generalisation of the proof presented in \cite[Theorem 3.5]{fr13} where $H_n^+(\Gamma)$ is replaced by $\wpr$ and relies on the simplicity of $\aut$ for a $\delta$-trace proved in \cite{bra13}.
\end{proof}

\subsection{Haagerup property}

The aim of this subsection is to prove that, under some hypothesis, the von Neumann algebra $L^\infty(H_{(B,\psi)}^{+}(\Gamma))$ has the Haagerup property.

\begin{prop}
Let $(B,\psi)$ be a finite dimensional C*-algebra endowed with a trace $\psi$ and $\Gamma$ be a finite group. Consider the free product decomposition of the von Neumann algebra $L^\infty(\cwpr)$ given in Remark \ref{decomp}. If for each $i$, $\dim(B_i)\geq 4$ then $L^\infty(\cwpr)$ has the Haagerup property.
\end{prop}

\begin{proof}
If $\psi$ is in particular a $\delta$-trace, i.e. the decomposition contains only one factor, the proof used in \cite[Subsection 3.3]{fr13} still applies.
The multifactor case is easily reduced to the previous one recalling that the Haagerup property is stable under tracial free products (see \cite[Proposition 3.9]{boc93}) and by using Remark \ref{decomp}.
\end{proof}

\bibliographystyle{amsalpha}
\bibliography{biblio}

\noindent
\textsc{{\footnotesize Lorenzo Pittau: Universit\'e de Cergy-Pontoise, 95000, Cergy-Pontoise, France\\
Univ Paris Diderot, Sorbonne Paris Cit\'e, IMJ-PRG, UMR 7586 CNRS, Sorbonne Universit\'es, UPMC Univ Paris 06, F-75013, Paris, France}}\\
\emph{{\footnotesize E-mail address: }}\texttt{{\footnotesize lorenzo.pittau@u-cergy.fr}}

\end{document}